\documentclass[letterpaper, 10 pt, conference]{ieeeconf}
\IEEEoverridecommandlockouts
\usepackage{graphics} 
\usepackage{epsfig} 
\usepackage{epstopdf}
\usepackage{amsmath} 
\usepackage{amssymb}  
\usepackage{mathtools}
\mathtoolsset{showonlyrefs}
\usepackage{bbm}
\usepackage{color}
\usepackage{cite}
\usepackage{float}
\usepackage{color}
\usepackage{comment}
\usepackage{subcaption}

\newtheorem{lemma}{Lemma}
\newtheorem{theorem}{Theorem}

\newtheorem{remark}{Remark}
\newtheorem{proposition}{Proposition}

\newcommand{\graph}[0]{{\mathcal{G}}}

\newcommand{\tr}{\mathbf{tr}}

\newcommand{\Nodes}{\mathcal{N}}
\newcommand{\Verts}{\mathcal{V}}

\newcommand{\Edges}{\mathcal{E}}

\newcommand{\Htwo}{{\mathcal{H}_2}}
\newcommand{\reals}{{\mathbb{R}}}
\newcommand{\sm}{\scalebox{0.75}[1.0]{$-$}}

\title{\LARGE \bf
Time Scale Design for Network Resilience
}

\author{Dillon R. Foight, Mathias Hudoba de Badyn, and Mehran Mesbahi
\thanks{The research of the authors was supported by the NSF GFRP under grant number DGE-1762114, NSERC under funding reference number CGSD2-502554-2017, the U.S. ARL and the U.S. ARO under contract number W911NF-13-1-0340, and the U.S. AFOSR under grant number FA9550-16-1-0022.~\copyright~2019~IEEE }%
\thanks{The authors are with the William E. Boeing Department of Aeronautics and Astronautics at the University of Washington, Seattle, WA 98195 USA. MHdB is also with the Automatic Control Laboratory at ETH, Z\"{u}rich, Switzerland. 
  e-mails: \texttt{\{dfoight,hudomath,mesbahi\}@uw.edu}.}%
}

\begin{document}

\maketitle
\thispagestyle{empty}
\pagestyle{empty}

\begin{abstract}
In this paper we consider the $\Htwo$-norm of networked systems with multi-time scale consensus dynamics. We develop a general framework for such systems that allows for edge weighting, independent agent-based time scales, as well as measurement and process noise. From this general system description, we highlight an interesting case where the influences of the weighting and scaling can be separated in the design problem. We then consider the design of the time scale parameters for minimizing the $\Htwo$-norm for the purpose of network resilience.
\end{abstract}


\section{INTRODUCTION}

Dynamical systems operating over networks appear in many natural and cyber-physical systems. A popular model of such dynamic processes is \emph{consensus}, which has been widely used for a variety of control and estimation applications, ranging from robotics and swarm deployment~\cite{Joordens2009,Hudobadebadyn2018}, distributed Kalman filtering~\cite{Olfati-Saber2005,HudobadebadynFillt2017}, and multi-agent systems~\cite{Chen2013a,Saber2003,Tanner2004}. A natural question in such scenarios is how the underlying network topology affects the behavior of the dynamics operating over the network. This question has attracted significant interest in systems and control communities, particularly as certain notions of performance and control can be directly related to graph theoretic properties of the network. Of particular interest for this work are system-theoretic measures such as $\Htwo$ and $\mathcal{H}_{\infty}$ system norms. For networked dynamical systems, the $\Htwo$-norm can be interpreted as a measure of how input energy is attenuated over the network, or how noise drives deviations from the natural consensus state~\cite{Siami2014a}. 

In light of these interpretations, there have been several works investigating the characterization of $\Htwo$ performance for consensus networks. In~\cite{Chapman2013a,Chapman2015,Hudobadebadyn2019}, the performance of leader-follower networks is considered, and algorithms for rewiring and reweighting the network for optimal noise rejection are discussed. Similarly,~\cite{Bamieh2012,Patterson2010,Patterson2014} have utilized the $\Htwo$-norm as a measure of coherence in networks and considered problems such as local feedback laws and leader selection to promote coherence. Most relevant to the present contribution, the works~\cite{Zelazo2011a,Zelazo2013} investigated the impact of cycles on the $\Htwo$ performance of noise-driven consensus networks. The examination of networks under noise inputs is important for real-world implementation of consensus onto physical systems, and for considering network resilience in the presence of adversarial noise injections.

A fully general model of performance of agents operating over a network should also include consideration of their individual dynamics. However, a common assumption within the consensus dynamics literature is that the agent dynamics are identical single or double integrators. This common model can be extended to encompass a class of heterogeneous agents by considering the case where individual agents' states evolve at differing rates. This is inherently a multi-time scale problem, and the analysis of such problems has historically offered techniques for formal description and controller synthesis for complex systems~\cite{Kokotovic1999}. Similar formulations arise in areas such as electrical networks~\cite{Dorfler2018} and power networks with generator inertia~\cite{Chakrabortty2013}. Thus, analysis of such multiple time scale models can increase the applicability of the consensus protocol to a wider range of real-world systems.
There is a growing body of literature that addresses the complications that naturally arise from the integration of multiple time scales into consensus, starting with the discussion of the consensus value for multi-rate integrators in~\cite{Olfati-Saber2004}. Issues such as convergence \cite{Pedroche2014}, stability \cite{Chapman2016,Awad2018}, controller design~\cite{Rejeb2016,Rejeb2018}, as well as single-influenced consensus performance~\cite{Foight2019} have since been addressed for such multi-scale networks. This existing literature has demonstrated that the inclusion of time scales into the consensus protocol can have a significant impact on the networked system, as well as shown that graph-theoretic interpretations of system-theoretic properties are not completely lost.

In this paper, we consider design problems for networked problems using the $\Htwo$ system norm as a metric of network's resiliency: a small $\Htwo$-norm characterizes a network that is resilient to external input. We consider a general formulation for single-integrator consensus that includes both edge weighting and nodal time scales, as well as process and measurement noises. Drawing from the work in~\cite{Zelazo2011a}, we transform the general consensus problem to one over the edge states, and consider design of the agent time scales with a focus on network resiliency. The main contributions of the paper are a similarity transformation yielding the dynamics of the edge states for scaled and weighted single integrator consensus, a method for separating the contributions of edge weighting and node time scales on the $\Htwo$ performance, and design problems for time scale assignment. 

The paper is organized as follows. In \S\ref{sec:math-prel}, we outline the notation and terminology used in the paper. We then introduce the problem setup in \S\ref{sec:setup}, followed by the main results of the $\Htwo$ performance metric formulation in \S\ref{sec:h2-results}, and the design problems for time scales in \S\ref{sec:ts-design} and \S\ref{sec:gu-design}. 

\section{MATHEMATICAL PRELIMINARIES}
\label{sec:math-prel}

Here, we provide a brief overview of the notation and terminology used throughout the paper, as well as relevant graph theoretic concepts. Column vectors are denoted as $x \in \reals^n$. Special vectors include the vector of all ones (zeros), $\mathbf{1}$ ($\mathbf{0}$), the vector of diagonal elements in a matrix, $\mathrm{diag}(M)$, and Euclidean basis vectors, $e_i \in \reals^n$, where the $i$ denotes the index of the non-zero element. Matrices will be denoted as $M \in \reals^{m \times n}$. The identity matrix will be denoted by $I$. Time-dependent quantities will be denoted as $x(t)$.

This paper considers dynamics governed by the interconnections of multi-rate, single integrator agents over connected, undirected, weighted communication graphs. In this formulation, we can consider a graph object defined by $\graph= \left(\Verts,\Edges\right)$, where $\Verts$ is the set of agents (nodes), $\Edges$ is the set of edges. Associated with each graph are $W$, a diagonal matrix of edge weights, and $E$, a diagonal matrix of node time scaling factors. Individual agents will be indexed by subscripts, e.g. $\nu_i \in \Verts$ to represent the $i$-th agent where $1\leq i \leq |\Verts|$. If $(i,j) \in \Edges$, the $i$-th and $j$-th agents are connected by an edge ($i \sim j$), and they are referred to as adjacent agents or neighbors.  For a given agent, $\nu_i$, $N(i) = \{j\ |\ i\sim j\ \forall j \in \Verts\}$ denotes the neighbors of $i$, and $\text{deg}(\nu_i) = |N(i)|$ denotes the unweighted degree of $i$. The edge set can be ordered by a mapping, $\kappa(\cdot)$, such that $l = \kappa(ij)$ if and only if $(i,j) \in \Edges$. By this mapping, we can denote the weight on edge $\kappa(ij)$ by $w_l$ or $w_{ij}$, interchangeably. The edge weights are assumed to be non-negative and symmetric, that is $w_{ij} = w_{ji}$. The scaling parameter of the $i$-th node is denoted by $\epsilon_i > 0$. The incidence matrix, $D(\graph)$ is a $|\Verts| \times |\Edges|$ matrix, where the $l$-th column denotes an edge between two nodes in the form of an edge vector, $a_{\kappa(ij)} = e_i - e_j$ (equivalently, $e_j - e_i$). Of particular interest will be the Laplacian-type matrices associated with the graph, which will be denoted by $L(\graph)$ or $L_e(\graph)$, where the subscript $e$ denotes an edge Laplacian. These matrices will be formally defined in \S\ref{sec:setup}. 


\section{PROBLEM SETUP} \label{sec:setup}

In this section, we describe a general formulation for consensus over a network with non-negative edge weighting, positive node time scaling, while accounting for possible measurement and process noise. The scaled consensus problem is derived from considering a group of $n$ multi-rate integrators~\cite{Olfati-Saber2004}, with zero-mean Gaussian process noise, $\omega_i(t)$ such that $\mathbf{E}\left[\omega(t) \omega(t)^T\right] = \mathrm{diag}(\sigma_{\omega_i}^2) $ for all $i \in \Verts$,
\begin{align}
\epsilon_i \dot{x}_i(t) =  u_i(t) + \omega_i(t), \label{eq:consensus}
\end{align}
where $x_i$ is the (scalar) state of the $i$-th agent, $\epsilon_i$ is the associated time scaling parameter, $u_i$ is the control input, and $\omega_i$ is the process noise that pollutes the control signal at the node level. A weighted, decentralized feedback controller that seeks to bring agents into consensus, but is impeded by measurement noise between adjacent agents, $v_{ij}(t)$ such that $\mathbf{E}\left[v(t) v(t)^T\right] = \mathrm{diag}(\sigma_{v_{ij}}^2) $ for all $(i,j) \in \Edges$, is given by,
\begin{align}
u_i(t) & = \sum_{j\in N(i)} \left[w_{ij} (x_j(t) - x_i(t)) + v_{ij}(t)\right] \nonumber \\
u(t) & = \sm D(\graph) W D(\graph)^T x(t) + D(\graph) v(t), \label{eq:control}
\end{align}
\noindent where $W$ is the matrix of edge weights with properties detailed in \S\ref{sec:math-prel}, $v(t)$ is the stacked vector of measurement noises, and $u(t)$ is the vector-valued input to all states. Applying~\eqref{eq:control} to the matrix version of~\eqref{eq:consensus} gives the general, time scaled and weighted consensus problem with process and measurement noise,
\begin{equation} \label{eq:model}
    \dot{x}(t) = \sm E^{\sm 1} L_w(\graph) x(t) + \begin{bmatrix}E^{\sm 1} & \sm E^{\sm 1} D(\graph)\end{bmatrix} \begin{bmatrix} \omega(t) \\ v(t) \end{bmatrix} \\
\end{equation}
\noindent where $L_w(\graph) = D(\graph) W D(\graph)^T$ is the weighted Laplacian matrix. Later in this section, we will consider two different options for output from~\eqref{eq:model}, which will allow us to assess how the available system output impacts the network performance. 

As noted by~\cite{Zelazo2011a}, for a connected graph, the zero eigenvalue of the Laplacian matrix precludes reasoning about the $\mathcal{H}_2$ performance of~\eqref{eq:model}. This property of the Laplacian matrix persists in the scaled, weighted case~\cite{Foight2019}, so as in~\cite{Zelazo2011a}, we will appeal to a similarity transformation that isolates the zero eigenvalue, presented in the following theorem.

\begin{theorem} \label{thm:similarity}
The scaled, weighted graph Laplacian for a connected graph, $\graph{}$, with time scale matrix $E$ and weight matrix $W$, given by $L_{w,s} = E^{\sm 1}D(\graph)W D(\graph)^T$, is similar to
\[
        \begin{bmatrix}
        L_{e,s} RWR^T & 0\\ 0 & 0
        \end{bmatrix},
\]
    where $L_{e,s} = D(\graph_\tau)^T E^{-1} D(\graph_\tau)$ is the edge Laplacian for a spanning tree $\graph_\tau$ which is symmetrically ``weighted'' by the time scaling parameters, and $R$ is the basis of the cut space of $\graph$ as defined as in~\cite{Zelazo2011a}: 
$
R(\graph) = \begin{bmatrix} I & T^{c}_\tau \end{bmatrix}
$
\noindent with,
\[T^c_\tau = (D(\graph_\tau)^T D(\graph_\tau))^{-1} D(\graph_\tau)^T D(\graph_c).\]
\noindent Here, the $\tau$ and $c$ subscripts denote the incidence matrices for a spanning tree and the complementary edges in $\graph$, respectively. 

\end{theorem}
\begin{proof}
Following \cite{Zelazo2011a}, we define the similarity transforms,
\begin{align*}
    S_v(\graph) &= \begin{bmatrix}
        E^{-1}D(\graph_\tau)\left( D(\graph_\tau) ^T E^{-1} D(\graph_\tau)\right)^{-1} & \mathbf{1}
    \end{bmatrix}\\
    S_v(\graph)^{-1} &= \begin{bmatrix}
        D(\graph_\tau)^T \\ \frac{1}{\epsilon_s} \mathrm{diag}(E)^T
    \end{bmatrix},
\end{align*}  
    where $\epsilon_s:=\sum_{i=1}^n \epsilon_i$ is the sum of the time scale parameters. Then, denoting $D_\tau:= D(\graph_\tau)$, $D:=D(\graph)$ (and adopting an analogous notation for other system matrices), we have,
    \begin{align*}
        &S_v^{-1} L_{w}(\graph) S_v\ \\
 &=        \begin{bmatrix}
        D_\tau^T E^{-1} D_\tau R W R^T D_\tau^T E^{-1} D_\tau \left( D_\tau E^{-1} D_\tau\right)^{-1} & 0\\ \frac{1}{\epsilon_s} \mathbf{1}^TDD^T D_\tau \left(D_\tau E^{-1} D_\tau\right)^{-1} & 0
        \end{bmatrix}\\
&=         \begin{bmatrix}
        L_{e,s} R W R^T & 0\\ 0 & 0
        \end{bmatrix},
    \end{align*}
    as desired.
\end{proof}

By noting that $S_v x_e(t) = x(t)$, the scaled, weighted consensus model with noise in~\eqref{eq:model} is equivalent to,
\begin{equation} \label{eq:edgemodel}
    \begin{aligned}
    \dot{x}_e(t) & = \begin{bmatrix} -L_{e,s}(\graph_\tau) R(\graph) W R(\graph)^T & 0 \\ 0 & 0 \end{bmatrix} x_e(t)  \\
    & \quad + \begin{bmatrix} D_{\tau}^T E^{-1} & -L_{e,s}(\graph_\tau) R(\graph) \\ \frac{\mathbf{1}^T}{\epsilon_s} & 0 \end{bmatrix} \begin{bmatrix} \omega(t) \\ v(t) \end{bmatrix},
    \end{aligned}
\end{equation}
\noindent where $L_{e,s}(\graph_\tau)$ is again the scaled edge Laplacian for a spanning tree $\graph_\tau$. We can note that the form of~\eqref{eq:edgemodel} naturally suggests a partitioning of the edge state variable into a set of states in the spanning tree and those in the consensus space (span$(\mathbf{1})$), $x_e(t) = \begin{bmatrix} x_\tau(t) & x_{\mathbf{1}}(t) \end{bmatrix}$. The resulting dynamics for the spanning tree states is taken from~\eqref{eq:edgemodel} as,
\begin{equation} \label{eq:treemodel}
\Sigma_\tau := \left\{
    \begin{aligned}
    \dot{x}_\tau (t) & = -L_{e,s}(\graph_\tau) R(\graph) W R(\graph)^T x_\tau (t) \\
    & \quad + D_{\tau}^T E^{-1} \Omega \hat{\omega} - L_{e,s}(\graph_\tau) R(\graph) \Gamma \hat{v} \\
    z(t) & = R(\graph)^T x_\tau(t),
    \end{aligned} \right.
\end{equation}
\noindent where $\hat{v}$ and $\hat{w}$ are normalized error signals, $\Omega = \mathbf{E}\left[w(t) w(t)^T\right]$, and $\Gamma = \mathbf{E}\left[v(t) v(t)^T\right]$. An important note is that the output of~\eqref{eq:treemodel} contains information of the cycle states due to the inclusion of $R(\graph)$ and the fact that the cycle states are linear combinations of the tree states~\cite{Zelazo2011a}. We can also consider the same edge state model with output given solely by the spanning tree states,
\begin{equation} \label{eq:treeonly}
\hat{\Sigma}_\tau := \left\{
    \begin{aligned}
    \dot{x}_\tau (t) & = -L_{e,s}(\graph_\tau) R(\graph) W R(\graph)^T x_\tau (t) \\
    & \quad + D_{\tau}^T E^{-1} \Omega \hat{\omega} - L_{e,s}(\graph_\tau) R(\graph) \Gamma \hat{v} \\
    z(t) & = x_\tau(t).
    \end{aligned} \right.
\end{equation}
The $\mathcal{H}_2$ performance of~\eqref{eq:treemodel} and~\eqref{eq:treeonly} are given by $\mathbf{tr}(R^T X^\star R)$ and $\mathbf{tr}(X^\star)$, respectively~\cite{skogestad2001}, where $X^\star$ is the positive-definite solution to the Lyapunov equation,
\begin{align}
& \sm L_{e,s}^\tau R W R^T X - X R W R^T L_{e,s}^\tau + D_{\tau}^T E^{\sm 1} \Omega \Omega^T E^{\sm 1} D_{\tau} + \nonumber \\
& \quad L_{e,s}^\tau R \Gamma \Gamma^T R^T L_{e,s}^\tau = 0. \label{eq:lyapunov}
\end{align}
In general, the addition of the weighting and scaling precludes a closed form solution to~\eqref{eq:lyapunov} (which is desirable to find $X$'s dependence on $E,W$), and numeric results yield a nonlinear mixing of weights and scaling parameters in the entries of $X$. However, in the following section we will outline a case when analytic solutions to~\eqref{eq:lyapunov} exist, providing insights for design of edge weights and scaling parameters for optimal performance. 

\section{Main Results}
\label{sec:results}


In this section, we investigate analytic results for the $\mathcal{H}_2$ performance of~\eqref{eq:treemodel} and~\eqref{eq:treeonly}, as well as design problems for time scale assignment.

\subsection{Analytic Solutions to the Lyapunov Equation} \label{sec:h2-results}

As previously noted, the inclusion of time scale parameters and weighting precludes a by-inspection solution for arbitrary covariances $\Omega$ and $\Gamma$. Let us investigate, then, the impact that the choice of covariance has on the performance of the system. We can observe that~\eqref{eq:treemodel} and~\eqref{eq:treeonly} have identical input matrices, $B:= [D_\tau^T E^{-1} \Omega\ -L_{e,s} R \Gamma]$. Also, recall that we can parameterize the $\Htwo$-norm in terms of the observability gramian and the input matrix for a system, $\Htwo^2 = \mathbf{tr}(B^T P_{O} B) = \mathbf{tr}(BB^T P_{O})$, where $P_O$ is the observability gramian for~\eqref{eq:treemodel} or~\eqref{eq:treeonly}~\cite{skogestad2001}. The observability gramian is independent of any choice of covariance; thus, in the following lemma we can bound the $\Htwo$ performance while isolating terms that depend on the covariances.

\begin{lemma} \label{h2bound}
Under the assumption that~\eqref{eq:treemodel} and~\eqref{eq:treeonly} are observable, the $\Htwo$ performance can be bounded by,
\[
\lambda_{\min}(BB^T) \tr(P_O) \leq \tr(B^T P_O B) \leq \lambda_{\max}(BB^T) \tr(P_O),
\]
where $B = [D_\tau^T E^{-1} \Omega\ -L_{e,s} R \Gamma]$ and $P_O$ is the observability gramian for~\eqref{eq:treemodel} or~\eqref{eq:treeonly}.
\end{lemma}
\begin{proof}
From the assumption of observability, $P_O$ is positive definite, and $B B^T$ is symmetric. Applying~\cite[Theorem 1]{fang1994} then gives the result.
\end{proof}

The above observation leads to a bound for the $\Htwo$ performance for any choice of covariances. In order to further separate out the effect of  covariances, however, we will find the following lemma useful, whose proof is omitted for brevity.

\begin{lemma} \label{minmax}
Given Hermitian $M \in \reals^{n\times n}$, and $Z \in \reals^{n\times m}$, if for $x \in \reals^{m}$, $Zx = \mathbf{0} \Rightarrow x = \mathbf{0}$, then,
\[\lambda_{\max}(Z^T M Z) \leq \lambda_{\max}(M) \lambda_{\max}(Z^T Z),\]
and
\[\lambda_{\min}(Z^T M Z) \geq \lambda_{\min}(M) \lambda_{\min}(Z^T Z).\]
\end{lemma}

Now, we can combine Lemmas~\ref{h2bound} and~\ref{minmax} to give a bound on the $\Htwo$ performance based on the properties of the covariance matrices. 

\begin{lemma} \label{fullbound}
Given observable edge consensus dynamics of the form~\eqref{eq:treemodel} or~\eqref{eq:treeonly}, the $\Htwo$ performance can be bounded by,
\begin{align*}
& \left[\lambda_{\min}(\Omega \Omega^T)\lambda_{\min}(B_\tau B_\tau^T) + \lambda_{\min}(\Gamma \Gamma^T)\lambda_{\min}(B_c B_c^T) \right]  \\
& \leq \frac{\Htwo^2}{P} \leq \left[\lambda_{\max}(\Omega \Omega^T)\lambda_{\max}(B_\tau^T) + \lambda_{\max}(\Gamma \Gamma^T)\lambda_{\max}(B_c^T) \right]
\end{align*}
where $B_\tau^T := D_\tau^T E^{-1}$, $B_c^T := L_{e,s} R$ and $P = \tr(P_O)$ is the trace of the associated observability gramian.
\end{lemma}
\begin{proof}
Expand $BB^T$ as,
\begin{align*}
BB^T & = D_\tau E^{-1} \Omega \Omega^T E^{-1} D_\tau^T + L_{e,s} R \Gamma \Gamma^T R^T L_{e,s} \\
& := B_{\tau}^T Q B_\tau + B_c^T G B_c,
\end{align*} 
where $Q = \Omega\Omega^T$, $G = \Gamma \Gamma^T$, $B_c = R^T L_{e,s}$, and $B_\tau = E^{-1} D_\tau^T$. Observe that $B_{\tau}^T Q B_\tau$ and $B_c^T G B_c$ are Hermitian. Thus, via Weyl's Inequality~\cite[Theorem 4.3.1]{Johnson1985},
\begin{align*}
\lambda_{\min}(B_{\tau}^T Q B_{\tau}) & + \lambda_{\min}(B_c^T G B_c) \\
& \leq \lambda_{\min} (B_{\tau}^T Q B_{\tau} + B_c^T G B_c), \\
\lambda_{\max} (B_{\tau}^T Q B_{\tau} & + B_c^T G B_c) \\
& \leq \lambda_{\max}(B_{\tau}^T Q B_{\tau}) + \lambda_{\max}(B_c^T G B_c).
\end{align*}
Lemma~\ref{minmax} can be applied to the individual $\lambda_{\min}(\cdot)$ and $\lambda_{\max}(\cdot)$ terms by noting that the spanning tree edge Laplacian has null space spanned by $\mathbf{0}$, so $B_c$ and $B_\tau$ satisfy the condition on $Z$ in Lemma~\ref{minmax}. Combining the resultant eigenvalue bounds with Lemma~\ref{h2bound} leads to the desired result.
\end{proof}

With Lemma~\ref{fullbound}, given any $\Omega$ and $\Gamma$, the $\Htwo$ performance can be bounded by the minimum and maximum eigenvalues of the covariances. Furthermore, if a convenient choice of $\Omega$ and $\Gamma$ exists, the bound illustrates that any other covariances sharing minimum and maximum eigenvalue properties will be covered by the same performance bounds. With that in mind, we propose the following covariance selection: $\Omega = \sigma_\omega E^{1/2}$ and $\Gamma = \sigma_v W^{1/2}$. By inspection,~\eqref{eq:lyapunov} then has the following solution,

\begin{equation} \label{eq:lysol}
X^\star = \frac{1}{2}\left(\sigma_w^2 (R W R^T)^{-1} + \sigma_v^2 L_{e,s}^\tau \right).
\end{equation}

Equation~\eqref{eq:lysol} is of particular interest as it shows that by associating the covariances with the magnitudes of the edge weights and the time scale parameters, the edge and node weightings are completely separated in their effect on the $\mathcal{H}_2$ performance, save for the placement of the $\sigma_{\omega}$ and $\sigma_v$ parameters. This choice does have the peculiar interpretation of seeming to remove the effects of the weighting and scaling, but the bound provided by Lemma~\ref{fullbound} allows this choice to apply to a wide range of more general covariance choices. Hereafter, then, we will investigate this separable solution. To aid in the consideration of these independent contributions later, we can define,
\begin{align}
\mathcal{H}_2(\Sigma) & = \frac{\sigma_{\omega}^2}{2} \mathbf{tr}(R^T (R W R^T)^{-1} R) + \frac{\sigma_v^2}{2} \mathbf{tr}(R^T L_{e,s}^\tau R) \nonumber \\
& := \mathcal{H}_2(\Sigma,W) + \mathcal{H}_2(\Sigma,E), \label{eq:cspart}
\end{align}
\noindent and similarly,
\begin{align}
\mathcal{H}_2(\hat{\Sigma}) & = \frac{\sigma_w^2}{2} \mathbf{tr}((R W R^T)^{-1}) + \frac{\sigma_v^2}{2} \mathbf{tr}(L_{e,s}^\tau) \nonumber \\
& := \mathcal{H}_2(\hat{\Sigma},W) + \mathcal{H}_2(\hat{\Sigma},E). \label{eq:tspart}
\end{align}
Furthermore, we can note that the tree-edge state and cycle state information can be separated. First, we start with the edge weight term,
\begin{align}
& \mathcal{H}_2(\Sigma,W) = \frac{\sigma_{\omega}^2}{2} \mathbf{tr}(R^T (R W R^T)^{-1} R) \nonumber \\
& = \frac{\sigma_{\omega}^2}{2} \mathbf{tr}\left(\begin{bmatrix}
I \\ \left(T^c_\tau)\right)^T
\end{bmatrix} (R W R^T)^{-1} \begin{bmatrix}
I & T^c_\tau 
\end{bmatrix} \right) \nonumber\\
& = \frac{\sigma_{\omega}^2}{2} \mathbf{tr}\left( \begin{bmatrix}
(R W R^T)^{-1} & (R W R^T)^{-1} T^c_\tau \\
(T^c_\tau)^T (R W R^T)^{-1} & (T^c_\tau)^T (R W R^T)^{-1} T^c_\tau 
\end{bmatrix} \right) \nonumber \\
& = \mathcal{H}_2(\hat{\Sigma},W) + \frac{\sigma_{\omega}^2}{2}\mathbf{tr}\left( (T^c_\tau)^T (R W R^T)^{-1} T^c_\tau \right). \label{eq:hatrelationW}
\end{align}
\noindent A similar relation can be found for the time scale term~\eqref{eq:tspart},
\begin{align}
\mathcal{H}_2(\Sigma,E) & = \frac{\sigma_v^2}{2} \mathbf{tr}(R^T L_{e,s}^\tau R) \nonumber \\
& = \frac{\sigma_v^2}{2} \mathbf{tr}\left( \begin{bmatrix}
L_{e,s}^\tau & L_{e,s}^\tau T^c_\tau \\
(T^c_\tau)^T L_{e,s}^\tau & (T^c_\tau)^T L_{e,s}^\tau T^c_\tau 
\end{bmatrix} \right) \nonumber \\
& = \mathcal{H}_2(\hat{\Sigma},E) + \frac{\sigma_v^2}{2} \mathbf{tr}\left( (T^c_\tau)^T L_{e,s}^\tau T^c_\tau \right). \label{eq:hatrelationE}
\end{align}
The simplifications to~\eqref{eq:cspart} and~\eqref{eq:tspart} show that the $\mathcal{H}_2$ performance of the $\Sigma$ system, which has output containing information from the cycle edge states, predictably contains the $\mathcal{H}_2$ performance for the $\hat{\Sigma}$ system as an isolated term. Taken together,~\eqref{eq:hatrelationW} and~\eqref{eq:hatrelationE} illustrate how the output information differences between~\eqref{eq:edgemodel} and~\eqref{eq:treemodel} influence the overall $\mathcal{H}_2$ performance for identical tree-edge-state dynamics. 

Finally, we can note that, similar to~\cite{Zelazo2011a}, the weighted cycles make closed-form solutions for terms containing $(R W R^T)^{\sm 1}$ difficult. However, analytic results are tractable in the case of tree graphs.


\subsubsection{Tree Graphs}
When the underlying graph topology is a tree, $R = I$, and~\eqref{eq:lysol} simplifies to,
\[
X^\star = \frac{1}{2}\left(\sigma_{\omega}^2 W^{\sm 1} + \sigma_v^2 L_{e,s}^\tau \right).
\]
\noindent Furthermore, in this case $\mathcal{H}_2(\Sigma_\tau) = \mathcal{H}_2(\hat{\Sigma}_\tau) = \mathbf{tr}(X^\star)$. A closed form solution for the performance in this case is given in the following lemma. 

\begin{lemma}\label{lem:tree} For a tree graph, the $\Htwo$ performance of the system is given by,
\begin{align}
\mathcal{H}_2(\Sigma_\tau) & = \frac{1}{2} \mathbf{tr}\left(\sigma_{\omega}^2 W^{-1} + \sigma_v^2 L_{e,s}^\tau \right) \nonumber \\
& = \frac{1}{2} \left( \sigma_{\omega}^2 \sum_{k=1}^{n-1} \frac{1}{w_{k}} + \sigma_v^2 \sum_{i = 1}^{n} \frac{\text{deg}(\nu_i)}{\epsilon_i} \right), \label{eq:treeresult}
\end{align}
\noindent where deg$(\nu_i)$ is the unweighted degree of agent $\nu_i$, and $k$ is an index over the edges. 
\end{lemma}
\begin{proof}
The first term follows from the fact that $W$ is a diagonal matrix of weights, so the trace of the inverse is simply the sum of the inverted weights. For the second term, consider one of the diagonal elements
$[L_{e,s}^\tau]_{kk} = a_k^T E^{-1} a_k = \epsilon_i^{-1} + \epsilon_j^{-1},$
\noindent where $a_k$ is the edge vector corresponding to the edge between nodes $i$ and $j$, that is, $k = \kappa(ij)$. Now consider a node $\nu_i$. In the sum over all edges of the graph, $\epsilon_i^{-1}$ will appear once for every edge that connects $\nu_i$ to its neighbors, which is the unweighted degree of $\nu_i$. Considering all other nodes yields the second term. 
\end{proof}







Lemma~\ref{lem:tree} shows a trade-off between the time scales and the network topology, determining the overall performance of the network. Also, note that for a given distribution of scaling parameters and edge weights, changing the assignment of edge weights does not affect the $\mathcal{H}_2$ performance, while the assignment of scaling parameters does.
This is in line with results in the context of single-input influenced consensus~\cite{Foight2019}, and shown in the following example. 

{\bf Example:} 
Consider the tree graph on six nodes in Figure~\ref{fig:tree}. Assume that we have some distribution of edge weights and node scaling parameters such that $\sum_i w_i^{-1} = 2\alpha$, $\sum_i \epsilon_i = 1$, and further, that $\epsilon_i \in (0.1,0.2,0.4)$. Also, let $\sigma_v = \sigma_{\omega} = 1$. 
\begin{figure}[H]
\centering
\includegraphics[scale=0.3]{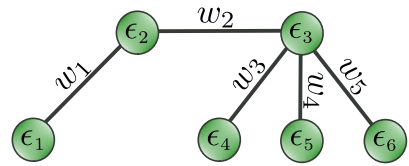}
\caption{Tree graph $\mathcal{T}$ for six agents. Each agent $i$ has an associated scaling parameter, $\epsilon_i$, and each $j$th edge is labeled $w_j$. Agent 3 has the highest degree.}
\label{fig:tree}
\end{figure}
Now, for any assignment of weights, the first term in~\eqref{eq:treeresult} will be $2 \alpha$, but the second term depends on the assignment of the scaling parameters. Consider the two following assignments: (1) set
$\epsilon_{1,4,5,6} = 0.1$, $\epsilon_2 = 0.2$, $\epsilon_3 = 0.4$.
In this case, the second term is equal to $\sum_{i=1}^{6} \text{deg}(i)/\epsilon_i = 60.0$ resulting in a combined $\mathcal{H}_2 = 30 + \alpha$. Now, consider a different distribution of scaling parameters.
(2) set $\epsilon_{1,3,4,5} = 0.1$, $\epsilon_2 = 0.2$, $\epsilon_6 = 0.4$. We can see that with agent 3 on a faster time scale, the second term suffers, $\sum_{i=1}^{6} \text{deg}(i)/\epsilon_i = 82.5$, resulting in a comparatively higher performance value of $\mathcal{H}_2 = 41.25 + \alpha$. 
Thus, we can see that high-degree nodes with slower time scales results in lower $\mathcal{H}_2$ performance. 
\subsection{Timescale Design For $\mathcal{H}_2$ Resilience} \label{sec:ts-design}

The results of the previous section suggest that a heuristic for minimizing the $\mathcal{H}_2$ performance is to assign slower timescales to high-degree agents. To investigate whether this holds in the presence of cycles, note that the minimization of $\mathcal{H}_2(\Sigma)$ can be formulated as a convex problem,

\begin{equation}
\begin{aligned}
\min_{\epsilon_{1}^{-1},\dots,\epsilon_n^{-1}}\ &\ \mathbf{tr}\left(R X R^T\right) \\
\text{s.t.}\ &\ {\epsilon}^{-1}_{\max} \leq {\epsilon_i}^{-1} \leq {\epsilon}^{-1}_{\min},\ \forall i \in \Nodes \\
&\ \mu \leq \sum_{i=1}^{n} {\epsilon_i^{-1}}, \\
&\ X = \frac{1}{2} \left((R W R^T)^{-1} + L_{e,s}^\tau \right),
\end{aligned} \tag{P1} \label{eq:convex}
\end{equation}

\noindent where we have taken the effective variances, $\sigma_{\omega}$ and $\sigma_v$, to be unity. The objective is convex in the optimization variables $(1/\epsilon_1,\dots,1/\epsilon_n)$, and the constraints are linear. The design parameter of $\mu$ serves to ensure that not all the agents can operate on the slowest time scale (the trivial solution).

We solved Problem~\eqref{eq:convex} on random graphs (with probability of an edge between any two nodes as 0.15) that featured multiple independent cycles. For $n=10$, $\epsilon_{\min} = 0.01$, $\epsilon_{\max} = 2.0$ and $\mu = 510.5$ (which can be interpreted as the value allowing up to a third of the nodes to be slow), the results and graph topology for one such graph are presented Figure~\ref{fig:cycleex}. These results (which appear to hold over a wide range of randomly generated graphs) suggest that the presence of cycles do not detract from the heuristic developed for tree graphs; that high degree nodes should be assigned slow time scales (high scaling parameters) to minimize $\mathcal{H}_2(\Sigma)$. 
\begin{figure}[H]
\centering
\vspace{-4mm}
\begin{subfigure}[b]{\linewidth}
\centerline{
\includegraphics[scale=0.35]{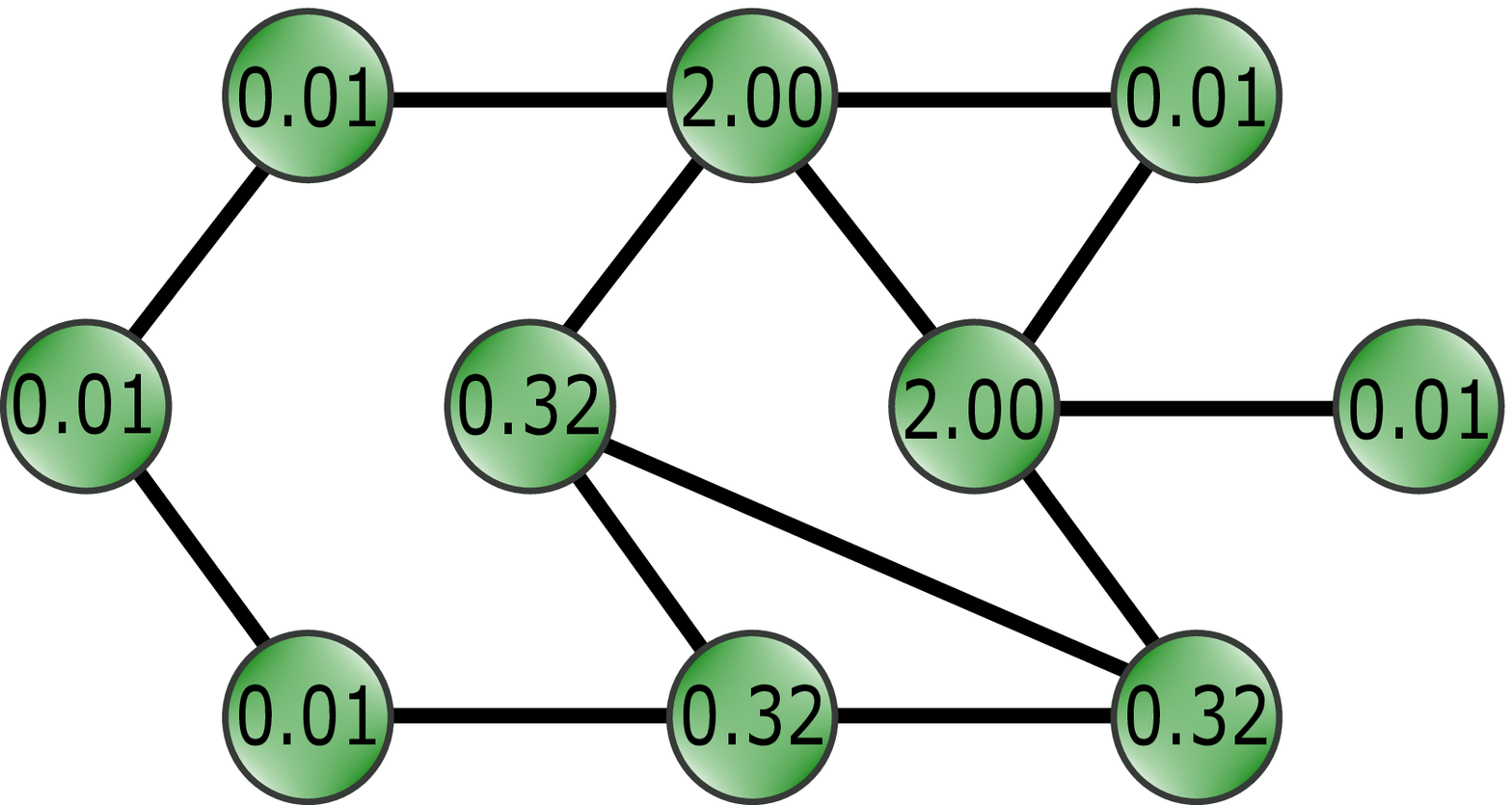}}
\caption{Random graph topology for 10 agents.}
\label{fig:cycleex}
\end{subfigure}

\begin{subfigure}[b]{0.48\linewidth}
\includegraphics[width=0.95\linewidth]{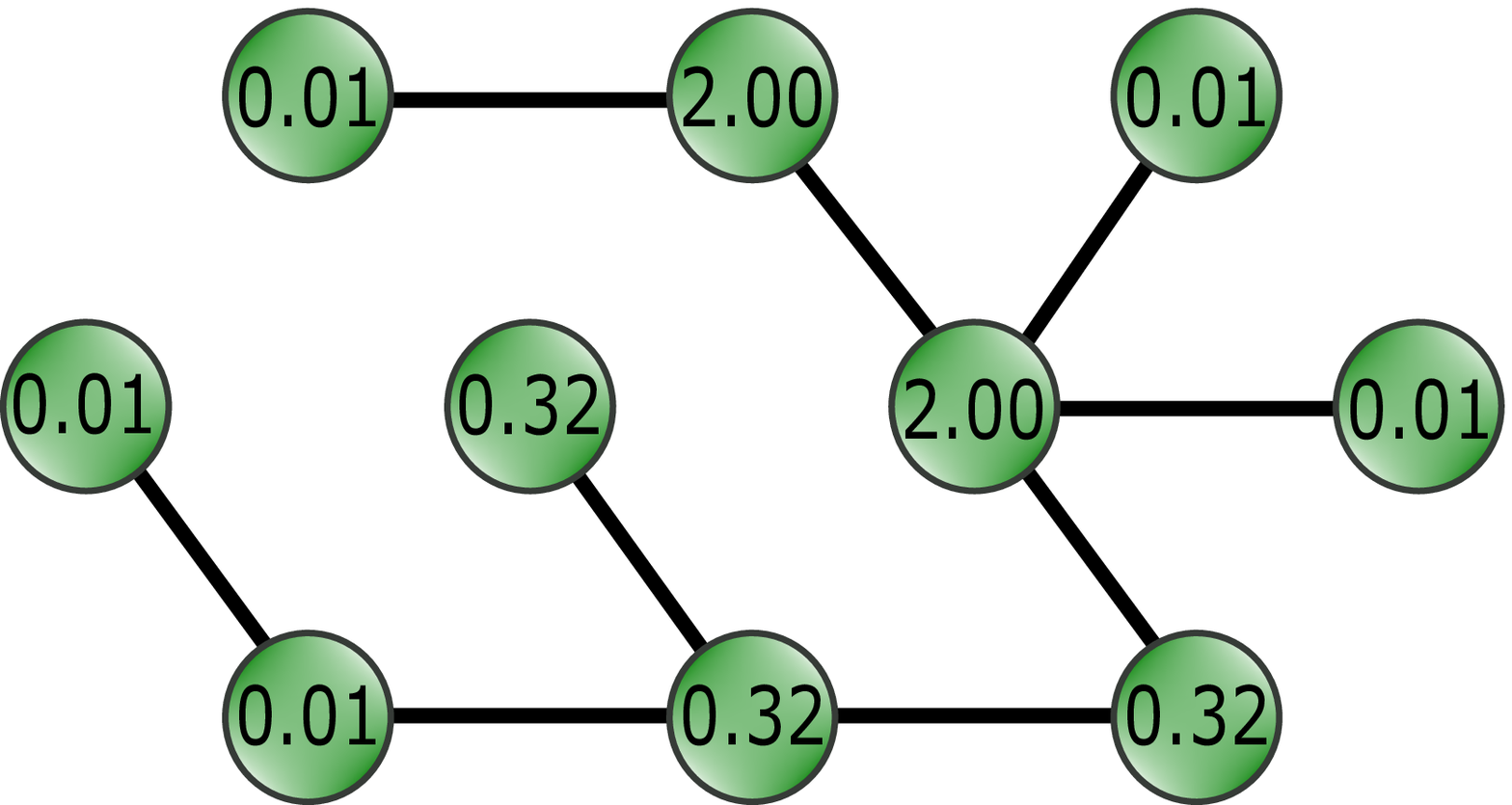}
\caption{Original spanning tree.}
\label{fig:cycleex_spanning1}
\end{subfigure}
\begin{subfigure}[b]{0.48\linewidth}
\includegraphics[width=0.95\linewidth]{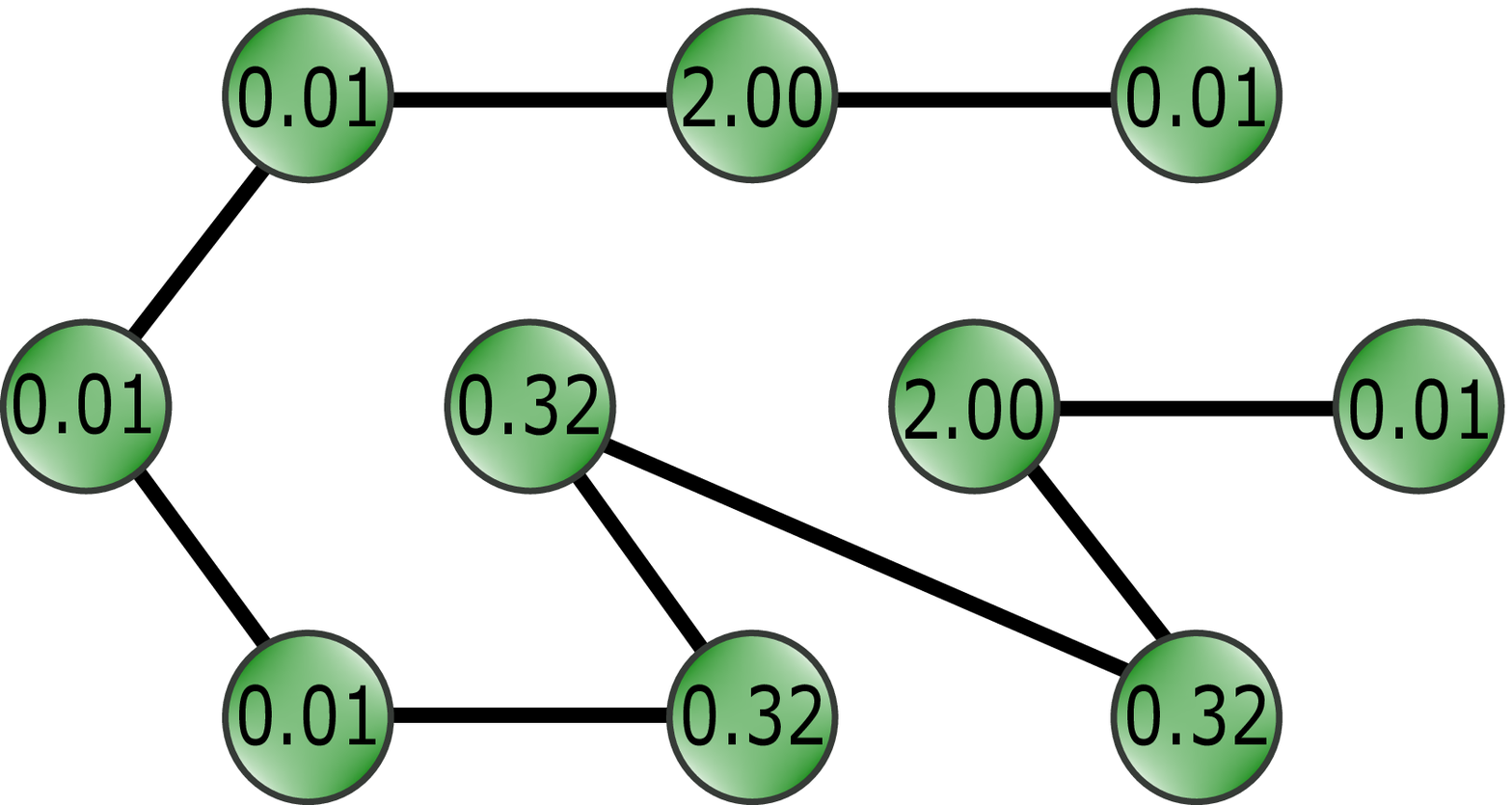}
\caption{Spanning path graph.}
\label{fig:cycleex_spanning2}
\end{subfigure}
\vspace{-1mm}
\caption{Time scale assignment by~\eqref{eq:convex} are printed in each node, showing the slowest time scales are assigned to nodes with highest degree.}
\end{figure}
A note of further interest is that, due to the inclusion of the cycle information in the output in~\eqref{eq:treemodel}, the optimal distribution from~\eqref{eq:convex} is independent of the selected spanning tree. We can see this by generating spanning trees with a variety of degree distributions, such as those in Figures~\ref{fig:cycleex_spanning1} and~\ref{fig:cycleex_spanning2}. However, we can recall from~\eqref{eq:hatrelationE} and~\eqref{eq:hatrelationW} that the $\Htwo$ performance for the $\Sigma$ system can be viewed as the performance for the $\hat{\Sigma}$ system with an additive term that encompasses the contribution of the cycle states. Consider then, the quantity
%
\begin{equation*}
K = \frac{\Htwo(\hat{\Sigma},E)}{\Htwo(\Sigma,E)},
\end{equation*}
\noindent which is a measure of how well the performance as measured by the spanning tree states represents the graph performance including cycle information. For the spanning tree in Figure~\ref{fig:cycleex_spanning1} we have $K \simeq 0.66$, and for the spanning tree in Figure~\ref{fig:cycleex_spanning2} we have $K \simeq 0.24$. Intuitively, this reflects that spanning trees which more accurately reflect the true degree distribution of the parent graph will have a higher $K$. In line with~\cite{Zelazo2011a}, this also shows a significant portion of the $\Htwo$ performance can come from the cycle contributions. In general, then, the performance of a given spanning tree may not be a good indicator of the full network performance, however, for graphs with few cycles, spanning trees that reproduce the full graph degree distribution closely can be a good approximation for the full network performance. 

The time scale assignment problem considered here demonstrates that while the heuristic developed from results on tree graphs appears to hold for more complex graph topologies, the performance of a given spanning tree does not necessarily reflect the performance of the full tree. In the next section, we will consider a reformulation of this problem that allows for an analytic result to the optimal assignment while also reformulating the performance constraint.


\subsection{Decentralized Updates for Optimal $\mathcal{H}_2$ Performance} \label{sec:gu-design}

In the previous section,~\eqref{eq:convex} included a design parameter to ensure that the trivial solution was avoided, however, computation of that parameter required complete knowledge of the global topology and time scale distribution. In response to an adversarial attack, such as malicious noise being injected into the system, it is of interest for the network to be able to quickly and autonomously adapt to minimize the effect of this influence. In light of this desire, consider~\eqref{eq:update},
\begin{align}
\begin{aligned}
\min_{\epsilon_1^{-1},\dots,\epsilon_n^{-1}}\ &\ \frac{1}{2} \mathbf{tr}\left(R^T L_{e,s}^\tau  R\right) + \frac{h}{2}\sum_{i=1}^n \epsilon_i^r \\
\text{s.t.}\ &\ {\epsilon}^{-1}_{\max} \leq {\epsilon_i}^{-1} \leq {\epsilon}^{-1}_{\min}\ \forall i \in \Nodes.
\end{aligned} \label{eq:update} \tag{P2} 
\end{align}
This is a minimization of the time scale portion of the separated $\mathcal{H}_2$ performance. In lieu of the sum constraint $\mu \leq \sum_i \epsilon_i^{-1}$, we introduce a regularization term $2^{-1}h \sum_{i=1}^n \epsilon_i^r$ which achieves a similar goal of penalizing large timescales for all nodes assuming positive, integer $r$. 

\begin{proposition}[Analytic Optimal Time Scale Assignment]
\label{prop:1}
Consider~\eqref{eq:update}. Let the region defined by the box constraints on $1/\epsilon_i$ be denoted by $\mathcal{C}$. Then, the minimizing assignment of time scale parameters is given by,
\[
\epsilon_i^\ast = \text{Proj}_\mathcal{C}\left[\left(\frac{\text{deg}(\nu_i)}{h r}\right)^{\frac{1}{r+1}} \right].
\]
\end{proposition}

\begin{proof}
Consider the cost function without the box constraint. Minimizing the cost alone can be achieved by setting its gradient equal to zero,
\begin{align*}
\frac{\partial f}{\partial \epsilon_i^{-1}} & = \frac{\text{deg}(\nu_i)}{2} - \frac{h r}{2} (\epsilon_i^{-1})^{-(r+1)}=0 \\
\epsilon_i^\ast & = \left(\frac{\text{deg}(\nu_i)}{h r}\right)^{\frac{1}{r+1}}.
\end{align*}
Projecting this result onto the constraint set gives the result.
\end{proof}

\begin{remark}
The assignment rule in Proposition~\ref{prop:1} is decentralized, as the optimal assignment value depends only on the (unweighted) degree of the $i$-th node and the parameters $h$ and $r$, which are locally known to the $i$-th node without global knowledge of the network topology.
\end{remark}

From this result we can see that for a class of regularization terms, the optimal time scale assignment is again driven by the degree distribution, which is in-line with the previous results. It is conceivable to consider using this result with online signal identification to locally adjust time scales in response to adversarial noise entering the system.

\section{Concluding remarks}

In this paper we have investigated how independent agent-based time scales can be designed to minimize the $\Htwo$-norm of consensus systems. We showed that for a convenient choice of noise covariances, the performance contributions of edge weights and time scaling are separable. This allowed for the independent consideration of time scale design for minimization of the $\Htwo$-norm. The contributions of this work have been an extension of previous methods into a framework which includes weighting and time scaling. We also identified a heuristic for the design of time scale parameters for network resilience, namely, that nodes of high degree have a large impact on the performance of the network and assigning them slow time scales can mitigate this effect. 

\section*{ACKNOWLEDGMENT}

This material is based upon work supported by the National Science Foundation Graduate Research Fellowship Program. Any opinions, findings, and conclusions or recommendations expressed in this material are those of the author(s) and do not necessarily reflect the views of the National Science Foundation.

\bibliographystyle{IEEEtran}
\bibliography{ref}

\end{document}